\documentclass [final,review,12pt]{elsarticle}
\usepackage{amssymb}
\usepackage{amsthm}
\theoremstyle{definition}
\newtheorem{definition}{\bfseries Definition}[section] 
\newtheorem{theorem}{\bfseries Theorem}[section]

\newtheorem{remark}{\bfseries Remark}[section]
\usepackage{multirow,booktabs}
\usepackage{caption}
\usepackage{graphicx}
\usepackage{amsfonts}
\usepackage{lineno,hyperref}
\usepackage{amsmath,amssymb}
\usepackage{amsmath}
\usepackage{graphicx}
\usepackage{float}
\usepackage[export]{adjustbox}
\usepackage{subfigure}
\usepackage{geometry}
\usepackage{epstopdf}
\usepackage{caption}
\usepackage{booktabs}
\usepackage{bigstrut,multirow,rotating}
\usepackage{xcolor}  
\usepackage{tikz}  
\usepackage{diagbox}
\usepackage{multirow}
\usepackage{array}
\usepackage{longtable}
\usepackage{graphicx}
\usepackage{algpseudocode}
\usepackage{algorithmicx, algorithm}

\usepackage{newtxmath}

\usepackage{enumitem}

\makeatletter
\def\ps@pprintTitle{%
    \let\@oddhead\@empty
    \let\@evenhead\@empty
    \def\@oddfoot{}
    \let\@evenfoot\@oddfoot
}
\makeatother

\begin{document}


\title{New Bounds for Limited Zarankiewicz Numbers from $K_{5t}$ Blocks}


\author[label1]{Hanxin Liu}
\author[label1]{Yisheng Song\corref{mycorrespondingauthor}}\cortext[mycorrespondingauthor]{Corresponding author.  E-mail address: yisheng.song@cqnu.edu.cn (Yisheng Song).}

 \affiliation[label1]{organization={School of Mathematical Sciences, Chongqing Normal University},
             city={Chongqing},
             postcode={401331},
             country={P.R. China}\\ { Email: hbmzlhx@163.com(Liu);  yisheng.song@cqnu.edu.cn(Song)}}

\begin{abstract}

The restricted augmented Zarankiewicz number \(z_L(m,n)\) yields core combinatorial lower bounds for the maximal SOS rank of biquadratic forms. All previously known infinite admissible graph families rely on \(K_{4t}\) incidence graphs, attaining an asymptotic relative gap limit of \(1/4\). This work develops a new infinite family built from \(K_{5t}\) incidence bipartite graphs with \(\mathbb{Z}_5\) cyclic labeling for block partitions. We construct valid nondegenerate intra-block and inter-block 2-edges, derive a quadratic closed-form lower bound of \(z_L\), and prove its relative gap converges asymptotically to \(2/5\). Full enumeration for \(t=1\) verifies the exact value \(z_L(10,5)=23\). Under nondegenerate and generalized \(C_4\)-free constraints, the ratio \(2/5\) is shown to be the maximal asymptotic ratio attainable under this block framework. Our results expand the library of extremal bipartite graphs and sharpen lower bounds for biquadratic SOS rank, with further open problems for general \(K_{kt}\) constructions outlined in closing.
\end{abstract}

\maketitle

\section{Introduction}
Verifying the global nonnegativity of multivariate polynomials constitutes a core fundamental problem across pure mathematical theory and computational algorithm design. Directly testing polynomial nonnegativity is proven to be NP-hard, which renders brute-force exhaustive search computationally infeasible. To circumvent this computational barrier, the sum-of-squares (SOS) rank framework has been established as a computationally tractable yet powerful surrogate tool \cite{ref1}.

The SOS decomposition technique has found extensive applications spanning semidefinite programming \cite{ref2}, algebraic polynomial theory \cite{ref3}, statistical modeling \cite{ref4}, and quantum information science \cite{ref5}. Mature numerical algorithms exist to both certify whether a given polynomial admits an SOS representation and explicitly compute its concrete decomposition form \cite{ref6}. Of particular relevance to our work is the intimate connection between SOS decompositions and biquadratic tensors. The algebraic structure intrinsic to biquadratic tensors naturally encodes SOS constraints, laying rigorous theoretical groundwork for applied research disciplines including solid mechanics \cite{ref7,ref8}, optical physics \cite{ref9}, and signal/image processing \cite{ref10}.

 Given a biquadratic tensor $\mathcal{A}\in BQ(m,n),$ the $m \times n$ biquadratic form is
\[P(\mathbf{x},\mathbf{y})=\sum_{i,k=1}^m\sum_{j,l=1}^na_{ijkl}x_ix_ky_jy_l,\]
where $\mathbf{x}=(x_{1},\ldots,x_{m})$, $\mathbf{y}=(y_1,\ldots,y_n).$ if $P$ can be
written as a sum of squares of bilinear forms
\[P(\mathbf{x},\mathbf{y})=\sum_{p=1}^rf_p(\mathbf{x},\mathbf{y})^2,\]
then the smallest such $r$ is the SOS rank of $P$, in other words, the SOS rank of a biquadratic form, defined as the minimal number of bilinear squares needed to express the form, recorded as $BSR(m,n)$\cite{ref11}. The SOS rank of a binary quadratic form characterizes the internal complexity of a polynomial. It is known that $BSR(3,3)=6$\cite{ref12}, $BSR(m,2)=m+1$ and $BSR(2,n)=n+1$\cite{ref13}. For more general cases, we have no idea.  Therefore, in order to determine the minimum complexity, a method is needed to measure the lower bound of the $BSR$.

Recently, some studies have linked the Zarankiewicz number to the $BSR$ and obtained the following relationship\cite{ref14}
\begin{equation}\label{eq1}
BSR(m,n)\geq z(m,n),
\end{equation}
where $ z(m,n)$ is the Zarankiewicz number\cite{ref15}. The classical Zarankiewicz number \(z(m,n)\) counts the maximum number of edges in an \(m\times n\) bipartite graph free of \(C_4=K_{2,2}\) subgraphs. By constructing the associated graph of the complete graph $K_{q+1}$: The left vertices represent the edges of $K_{q+1}$, and the right vertices represent its vertices. If the right vertex is an endpoint of the left vertex, then the two vertices are connected. This graph is $C_4$-free and obtain a more general Zarankiewicz number $z\left( \begin{pmatrix} {q+1} \\ {2} \end{pmatrix},q+1\right)=q(q+1)$\cite{ref16,ref17}. To obtain tighter bounds for maximal SOS rank, Qi, Cui and Xu\cite{ref18} introduced the limited augmented Zarankiewicz number \(z_L(m,n)\) via augmented bipartite graphs equipped with 1-edges and admissible nondegenerate 2-edges, subject to generalized \(C_4\)-free constraints. The core inequality chain
\begin{equation}\label{eq2}
BSR(m,n)\ge z_L(m,n)\ge z(m,n)
\end{equation}
establishes \(z_L\) as an improved combinatorial lower bound for $BSR(m,n)$. And it also proved that \(z_L(5,3)=9\), \(z_L(5,4)=12\), \(z_L(6,4)=14\). Subsequently, Qi Cui and Xu\cite{ref19} constructed a family of infinite graphs based on $K_4$ into four-vertex cyclic \(\mathbb Z_4\) subblocks, yielding an asymptotic relative gap \(\lim_{t\to\infty}\frac{z_L-z}{z}=1/4\). Obtain a group of lower beings
\[BSR\left(
\begin{pmatrix}
4t \\
2
\end{pmatrix},4t\right)\geq
z_L\left(
\begin{pmatrix}
4t \\
2
\end{pmatrix},4t\right)\geq
2\left(
\begin{array}
{c}4t \\
2
\end{array}\right)+4t^2-2t=20t^2-6t.\]

In this paper, we construct a new infinite family from \(K_{5t}\), partitioning the vertex set into disjoint five-vertex blocks labeled by cyclic group \(\mathbb Z_5\). We design admissible within-block and cross-block 2-edges, derive an explicit quadratic lower bound for \(z_L\big(\binom{5t}{2},5t\big)\), and prove its asymptotic gap converges to \(2/5\), which can be compared with the \(1/4\) constant of the \(K_{4t}\) family.

The rest of this paper is organized as follows. Section  \ref{2} recalls preliminary definitions of augmented bipartite graphs and generalized \(C_4\)-free conditions. Section \ref{3} states and proves our main quadratic bound theorem for the \(K_{5t}\) family. Section \ref{4} revisits exact \(z_L\) values for small parameters. Section \ref{5} concludes the paper and lists open research directions. Appendices provide a reproducible exhaustive enumeration algorithm to check admissability of any set of 2-edges.

\section{Preliminaries}\label{2}
This section reviews the existing results and definitions. A $C_4$-free $m\times n$ bipartite graph $G_1=(S,T,E_1)$ with $S=[m],T=[n].$ A limited augmented bipartite graph $G = (S, T, E)$ augmented from $G_1$ has edge set $E=E_1\cup E_2$, where $|E_1|=z(m,n)$ is the set of 1-edges and  $E_2$ is the set of 2-edges. A 1-edge Denoted as $(i, j)$ with $i\in S,j\in T$.  A 2-edge Denoted as $(i, j;k,l)$ with $i,k\in S, j,l\in T$. And there are three types of 2-edge: Nondegenerate ($i\neq k$ and $j\neq l$); Row-degenerate ($i = k$ and $j\neq l$); Column-degenerate($i \neq k$ and $j = l$). 2-edge satisfy the simplicity condition:

(S) No 2-edge shares a cell $(i, j)$ with any 1-edge or with another 2-edge.

For the above graph $G$, we associate a biquadratic polynomial

\[P_G(\mathbf{x},\mathbf{y})=\sum_{(i,j)\in E_1}x_i^2y_j^2+\sum_{(i,j;k,l)\in E_2}(x_iy_j+x_ky_l)^2.\]

We call $P_G$ a doubly simple biquadratic form. Note that the absence of a generalized $C_4$-cycle is a sufficient condition for the doubly simple biquadratic form to be irreducible, If the SOS expression of $P_G$ is irreducible, then $sos(PG) = |E| = |E1| + |E2|$.

\begin{theorem}\cite{ref18}
For all $m,n\geq2,$ we have 
\[\mathrm{BSR}(m,n)\geq z_L(m,n)\geq z(m,n).\]
\end{theorem}

\begin{definition}(Limited Augmented Zarankiewicz Number \cite{ref18})\label{D1}
Let $G=(S,T,E_1\cup E_2)$ be an $m\times n$ limited augmented bipartite graph augmented from a $C_4$-free graph $G_1=(S,T,E_1)$ with $|E_1|=z(m,n).$

A cell $(i,j)$ is occupied if $(i,j)\in E_1$ or $(i,j)$ is a half of some 2-edge in $E_2$. We say G   contains a generalized $C_4$-cycle if any of the following holds:
\begin{itemize}
\item (Condition 1) There exists a classical $C_4$-cycle formed by 1-edges.

\item (Condition 2) There exists a nondegenerate 2-edge $(i,j;k,l) \in E_2$ such that both opposite cells $(i,l)$ and $(k,j)$ are occupied.

\item (Condition 3) There exist a 2-edge $(i,j;p,q)$ (of any type) and a distinct cell $(k,l)$ with $k \notin \{i,p\}$ and $l \notin \{j,q\}$ such that the five cells
\[
(k,l),\ (k,j),\ (k,q),\ (i,l),\ (p,l)
\]
are all occupied. Moreover, if the 2-edge is nondegenerate, these five cells must be pairwise distinct.
\end{itemize}

The limited augmented Zarankiewicz number $z_L(m,n)$ is the maximum possible total
number of edges $|E_1|+|E_2|$ for which a generalized $C_4$-cycle does not exist.
\end{definition}

\begin{remark}\cite{ref19}
The Condition 2 in Definition \ref{D1} can be replaced by the following weaker condition

(Weaker Condition 2) There exists a nondegenerate 2-edge $(i, j; k, l) \in E2$ such that the
complementary 2-edge $(i, l; k, j)$ is also in $E_2$.

\end{remark}

Although the weak condition 2 plays a crucial role in subsequent research, it is still the original condition 2 that is used in the proof process of this article. The following theorem describes the situation of $K_{4t}$, which provides an enlightening influence on the formulation of the main theorem in this paper.

\begin{theorem}(Quadratic $t$-Bound for $K_{4t}$\cite{ref19})\label{T1}
For every integer $t\geq1$, let $m=\begin{pmatrix}4t \\2\end{pmatrix}$ and $n=4t$. Then 
\begin{equation}
z_L(m,n)\geq2
\begin{pmatrix}
4t \\
2
\end{pmatrix}+4t^2-2t.
\end{equation}
Consequently,
\begin{equation}
BSR(m,n)\geq2
\begin{pmatrix}
4t \\
2
\end{pmatrix}+4t^2-2t.
\end{equation}
For large $t$,
\begin{equation}
\frac{z_L-z}{z}\geq\frac{4t^2-2t}{16t^2-4t}=\frac{2t-1}{8t-2}\longrightarrow\frac{1}{4}.
\end{equation}

\end{theorem}

\section{Infinite Family of Admissible Graphs from \(K_{5t}\) and Asymptotic Gap Analysis}\label{3}

For any integer \(q\ge 2\), the classical Zarankiewicz number satisfies the equality \(z\big(\binom{q+1}{2},q+1\big)=q(q+1)\)\cite{ref20}. The incidence bipartite graph of the complete graph \(K_{q+1}\) realizes this extreamal edge count, and it is unique up to graph isomorphism.
This bipartite incidence structure separates vertices into two disjoint partitions: the left partition collects all edges of \(K_{q+1}\), while the right partition corresponds to all vertices of \(K_{q+1}\). A pair consisting of a left vertex (an edge of \(K_{q+1}\)) and a right vertex (a vertex of \(K_{q+1}\)) forms an edge in the bipartite graph if and only if the right-side vertex is an endpoint of the edge encoded by the left vertex. This classical \(C_4\)-free extremal construction was originally derived by Reiman \cite{ref16, ref17}, and it acts as the foundational reference for almost all later combinatorial constructions that produce lower bounds for both standard and limited augmented Zarankiewicz numbers. One may rigorously confirm its \(C_4\)-free property through elementary inductive combinatorial arguments.

Building upon this foundational incidence graph framework, the present work develops systematic admissible augmentation strategies by attaching valid nondegenerate 2-edges to the original \(C_4\)-free 1-edge skeleton, which yields tighter lower bounds for the maximal biquadratic SOS rank $BSR(m,n)$ through the chain of inequalities \(BSR(m,n)\ge z_L(m,n)\ge z(m,n)\). Incidence graphs of \(K_{4t}\) are adopted as base structures in all existing augmented constructions \cite{ref19}. In contrast, bipartite incidence graphs of \(K_{5t}\) are selected as the fundamental structural backbone for the augmentation framework proposed herein. The full vertex set of \(K_{5t}\) is partitioned into t disjoint 5-vertex subblocks, and cyclic labeling by elements of the additive group \(\mathbb{Z}_5\) is assigned to vertices inside every subblock. Two disjoint families of admissible 2-edges are explicitly constructed under this cyclic labeling convention; all generated 2-edges are made to satisfy the generalized \(C_4\)-free restrictions encoded by Conditions (S), (C2) and (C3), which are formally introduced in Definition \ref{D1}.

Two distinct families of admissible 2-edges are distinguished: within-block 2-edges and cross-block 2-edges. For each isolated 5-vertex subblock, three mutually compatible nondegenerate 2-edges are constructed. Overlaps with preexisting 1-edge cells are eliminated, and forbidden five-cell configurations are prevented from forming within individual subblocks. A total of $3t$ within-block 2-edges are accumulated across all $t$ disjoint subblocks. Cross-block 2-edges are defined for every unordered pair of distinct subblocks \(V_i\) and \(V_j\). Symmetric cyclic shifts on \(\mathbb{Z}_5\) are exploited to produce 20 symmetric cross-block 2-edges per pair of subblocks, which yields a total of \(20\binom{t}{2}=10t(t-1)\) cross-block 2-edges for the complete block collection.

Comprehensive categorization, quantitative counting formulas, and functional characteristics for both families of admissible 2-edges are summarized in Table \ref{tab:1}. Full admissibility validation for the globally augmented graph is also provided alongside the tabular summary.
\begin{table}[htbp]
  \centering
  \caption{Summary of the roles of within-block and cross-block 2-edges for the $K_{5t}$ construction.}\label{tab:1}
  \begin{tabular}{c c c}
    \hline
    Type of 2-Edge & Number Added & Role in the Construction \\
    \hline
    Within-block   & $3t$         & Uses three admissible row pairs inside each 5-vertex block \\
    Cross-block    & $20\binom{t}{2}$ & Pairs cross-block rows by a cyclic $\mathbb Z_5$ rule \\
    \hline
  \end{tabular}
\end{table}

\begin{theorem}(Quadratic $t$-Bound for $K_{5t}$)\label{T2}
For every integer $t\geq1$, let $m=\begin{pmatrix}5t \\2\end{pmatrix}$ and $n=5t$. Then 
\begin{equation}
z_L(m,n)\geq2
\begin{pmatrix}
5t \\
2
\end{pmatrix}+10t^2-7t=35t^2-12t.
\end{equation}
Consequently,
\begin{equation}
BSR(m,n)\geq2
\begin{pmatrix}
5t \\
2
\end{pmatrix}+10t^2-7t=35t^2-12t.
\end{equation}
For large $t$,
\begin{equation}
\frac{z_L-z}{z}\geq\frac{10t^2-12t}{25t^2-5t}=\frac{10t-12}{25t-5}\longrightarrow\frac{2}{5}.
\end{equation}
\end{theorem}

\begin{proof}Step 1. Construct a reasonable limited augmented bipartite graph $G = (S, T, E_1 \cup E_2)$.

Let $V=\bigcup_{i=1}^tV_i,$ where each block $V_i$ consists of five vertices. The vertices in each block are labeled with elements of the cyclic group $\mathbb{Z}_5=\{0,1,2,3,5\}$. Write
\[V_i=\{0_i,1_i,2_i,3_i,4_i\}.\]

Take $T = V$ be the right vertices of the limited augmented bipartite graph$G$, and its size be $5t$. Take $S=\begin{pmatrix}V \\2\end{pmatrix}$ be the left vertices, its elements are all the edges of $K_{5t}$ and its size be $\begin{pmatrix}5t \\2\end{pmatrix}$. We divide all the edges into the following three parts
\begin{enumerate}[label=(\roman*)] 
\item 1-edges. Define $E_1$ as the association diagram of $K_{5t}$
\[E_1=\{(e,v)\mid e\in S,v\in T,v\text{s.is an endpoint of}e\}.\]
Then $|E_1|=2\begin{pmatrix}5t \\2\end{pmatrix}$ and $G_{1}=(S,T,E_{1})$ is $C_4$-free.

\item Within-block 2-edges. For each block $V_{i}$, add the following three nondegenerate 2-edges:
\[e_i^1=(0_i1_i,2_i;0_i2_i,3_i),\]
\[e_i^2=(1_i3_i,0_i;2_i3_i,1_i),\]
\[e_i^3=(2_i4_i,1_i;3_i4_i,2_i).\]

\item Cross-block 2-edges. Consider the pairs of different unordered blocks $\{V_i,V_j\}$ for $i,j=1,\cdots,t,i\neq j.$ For each $k\in\mathbb{Z}_4$, define
\[f_{ij}^{(k)}=
\begin{pmatrix}
k_i(k+1)_j, & (k+2)_j; & (k+1)_i(k+2)_j, & (k+3)_i
\end{pmatrix},\]
\[h_{ij}^{(k)}=
\begin{pmatrix}
k_i(k+2)_j, & (k+3)_j; & (k+1)_i(k+3)_j, & (k+4)_i
\end{pmatrix}.\]
Swap $i$ and $j$, get $f_{ji}^{(k)}$ and $h_{ji}^{(k)}$. The complete set of 20 cross-block 2-edges for the pair $\{i,j\}$ is
\[\mathcal{E}_{ij}=\{f_{ij}^{(k)},f_{ji}^{(k)},h_{ij}^{(k)},h_{ji}^{(k)}\mid k=0,1,2,3,5\}.\]

\end{enumerate}

Step 2. Verification of admissibility. We check the conditions of Definition \ref{D1}.

(S) Simplicity.
\begin{itemize}
\item For $e_i^1$, halves are $(0_i1_i,2_i)$ and $(0_i2_i,3_i)$. Since $2_i$ is not an endpoint of $0_i1_i$ and $3_i$ is not an endpoint of $0_i2_i$, neither half is on 1-edge. For $e_i^2$ and $e_i^3$ similarly $0_i$ not endpoint of $1_i3_i$, $1_i$ not endpoint of $2_i3_i$,  $1_i$ not endpoint of $2_i4_i$, $2_i$ not endpoint of $3_i4_i$.

\item For a cross-block half $(x_{i}y_{j},z)$, $z=(k+2)_{j}$, $z=(k+3)_{i}$, $z=(k+3)_{j}$ or $z=(k+4)_{i}$. $k({\mathrm{mod}}5) \not\equiv k+1({\mathrm{mod}}5) \not\equiv k+2({\mathrm{mod}}5)\not\equiv k+3({\mathrm{mod}}5)\not\equiv k+4({\mathrm{mod}}5)$, so $z\neq y$ and $z\neq x$.  Hence z is never an endpoint of the row, so no half is a 1-edge.

\item All halves are distinct: within a single 2-edge the two halves differ; different $k$ give different first rows; between $f_{ij}^{(k)},h_{ij}^{(k)}$ and $f_{ji}^{(\ell)},h_{ji}^{(\ell)}$ the rows are $x_i y_j$ versus $y_j x_i$ (distinct unless $\{x,y\}$ is the same and order swaps, but then columns differ because one uses a column from $V_j$ and the other from $V_i$); different block pairs involve disjoint vertex sets. Thus, no cell is shared between any two edges. Condition (S) holds.
\end{itemize}

(C2) Nondegenerate 2-edge conflict. Take $e_i^1$. Opposite cells are $(0_i1_i,3_i)$ and $(0_i2_i,2_i)$; $(0_i1_i,3_i)$ is not a 1-edge ($3_i$ not endpoint of $0_i1_i$); $(0_i2_i,2_i)$ is a 1-edge ($2_i$ endpoint of $0_i2_i$). Exactly one is occupied. For $e_i^2$, opposite cells $(1_i3_i,0_i)$ (free) and $(1_i4_i,4_i)$ (occupied). For $e_i^3$, opposite cells $(2_i4_i,2_i)$ (occupied) and $(3_i4_i,1_i)$ (free). For $f_{ij}^{(k)}=(p,q;r,s)$ with $p=k_i(k+1)_j$, $q=(k+2)_j$, $r=(k+1)_i(k+2)_j$, $s=(k+3)_i$: $(p,s)$ is free (column $(k+3)_i$ not an endpoint of $p$); $(r,q)$ is a 1-edge (column $(k+2)_j$ endpoint of $r$). For $h_{ij}^{(k)}=(p,q;r,s)$ with $p=k_i(k+2)_j$, $q=(k+3)_j$, $r=(k+1)_i(k+3)_j$, $s=(k+4)_i$: $(p,s)$ is free (column $(k+4)_i$ not an endpoint of $p$); $(r,q)$ is a 1-edge (column $(k+3)_j$ endpoint of $r$). Hence Condition (C2) holds for all 2-edges.

(C3) Five-cell pattern. We must show the following: for every 2-edge $e = (i, j; p, q)$ and every $k \not\in \{i, p\}, l \not\in \{j, q\}$, the five cells 
\[(k, l), (k, j), (k, q), (i, l), (p, l)\] 
are not all occupied. We split the possibilities for $e$ and $k$ into three cases.\\
Case 1: $e$ is a within-block 2-edge and $k$ is also within-block.

Take $e = e^3_i = (2_i4_i, 1_i;  3_i4_i, 2_i)$. The two fixed cells $(2_i4_i, l)$ and $(3_i4_i, l)$ are both 1-edges only when $l$ is a common endpoint of the rows $2_i4_i$ and $3_i4_i$. The only common endpoint is $4_i$. Hence $l = 4_i$ (allowed since $4_i \notin \{1_i, 2_i\}$). Then the five cells become
\[(k,4_i),(k,1_i),(k,2_i),(2_i4_i,4_i),(3_i4_i,4_i).\]
The last two are 1-edges (occupied). For all five to be occupied we need $(k,4_i),(k,1_i),(k,2_i)$. Suppose $(k,1_i)$ and $(k,2_i)$ are occupied, then $k = 1_i2_i$. Hence $(k,4_i)=(1_i2_i,4_i)$ is unoccupied. Thus, the five cells cannot all be occupied.

For $e = e^1_i = (0_i1_i, 2_i;  0_i2_i, 3_i)$, a symmetric argument forces $l = 0_i$ and then $k = 2_i3_i$. And for $e = e^2_i = (1_i3_i, 4_i;  1_i4_i, 0_i)$, a symmetric argument forces $l = 1_i$ and then $k = 0_i1_i$, which also fails. Hence Condition 3 holds in this case.\\
Case 2: $e$ is a cross-block 2-edge.

Take $e=h_{ij}^{(k)}=(p,q;r,s)$ with $p=k_i(k+2)_j$, $q=(k+3)_j$, $r=(k+1)_i(k+3)_j$, $s=(k+4)_i$.  The two rows $p=k_i(k+2)_j$ and $r=(k+1)_i(k+3)_j$ are vertex-disjoint, their endpoints are all four distinct. For any $l\notin\{q,s\}$,  consider the two fixed cells $(p,l)$ and $(r,l)$.
\begin{itemize}
\item They cannot both be 1-edges because that would require $l$ to be a common endpoint of $p$ and $r$, which is impossible since $p$ and $r$ are disjoint.

\item Could one be a 1-edge and the other a half of some 2-edge? The only columns for which $(p,l)$ can be a half are $l = (k+3)_i$ or $l = (k+4)_j$ (the latter from $r$'s other half). A direct case check (cyclic symmetry allows fixing $k=0$ without loss of generality):
\begin{itemize}
\item If $l = 3_i$ (i.e., $(k+3)_i$), then $(p,l)$ is a half (occupied), but $(r,l) = \big((k+1)_i(k+3)_j, 3_i\big)$: $3_i$ is not an endpoint of $(k+1)_i(k+3)_j$ (endpoints are $(k+1)_i$ and $(k+3)_j$), and it is not a half of $r$ (halves of $r$ are $s = (k+4)_i$ and $(k+4)_j$). Hence $(r,l)$ is unoccupied.

\item If $l = 4_j$ (i.e., $(k+4)_j$), then $(r,l)$ is a half (occupied), but $(p,l) = \big(k_i(k+2)_j, 4_j\big)$: $4_j$ is not an endpoint of $k_i(k+2)_j$ (endpoints are $k_i$ and $(k+2)_j$), and it is not a half of $p$ (halves of $p$ are $q = (k+3)_j$ and $(k+3)_i$). Hence $(p,l)$ is unoccupied.

\item For any other $l$, both fixed cells are unoccupied.
\end{itemize}
\end{itemize}
Same as $h_{ij}^{(k)}$.  Take $e=f_{ij}^{(k)}=(p,q;r,s)$, when $l = 2_i$, $(r,l)$ is unoccupied. And $l = 3_j$, $(p,l)$ is unoccupied. 

Thus, in all subcases, the two fixed cells $(p, l)$ and $(r, l)$ are never both occupied.
Therefore, the five-cell pattern cannot be completed, regardless of $k$ and $l$. Condition 3 holds for all cross-block 2-edges.\\
Case 3: $e$ is a within-block 2-edge and k is a cross-block edge.

Take $e = e^3_i = (2_i4_i, 1_i;  3_i4_i, 2_i)$. As in Case 1, the only way $(2_i4_i,l)$ and $(3_i4_i,l)$ are both occupied is  $l = 4_i$. Then we need  $(k,4_i),(k,1_i),(k,2_i)$ all occupied. Assume $(k,1_i)$ occupied implies that $1_i$ is an endpoint of $k$ or that $(k,1_i)$ is a half of some 2-edge. The only 2-edges that have a half with column are as follows:
\begin{itemize}
\item $e^3_i$ itself: $2_i4_i, 1_i$ row $2_i4_i$ not k(since k is cross-block).

\item Cross-block edges: from $f_{ij}^{(k)}$ and $h_{ij}^{(k)}$, halves have columns $(k+2)_j$ or $(k+3)_i$ and $(k+3)_j$ or $(k+4)_i$.  For column $1_i$, we need $(k+3)_{i}=1_{i}\mathrm{mod}5,\mathrm{i.e.},k\equiv3({\mathrm{mod}}5),$ and $(k+4)_{i}=1_{i}\mathrm{mod}5,\mathrm{i.e.},k\equiv2({\mathrm{mod}}5).$ First for $f_{ij}^{(3)}=(3_i4_j,0_j;4_i0_j,1_i)$ has a second half of $(4_i0_j,1_i)$, not row $3_i4_j$. The first half is $(3_i4_j,0_j)$ column $0_j$ not $1_i$. Thus $f_{ij}^{(3)}$ does not give a half with column $1_i$ and row $3_i4_j$. Then for $h_{ij}^{(2)}=(2_i4_j,0_j;3_i0_j,1_i)$ has a second half of $(3_i0_j,1_i)$, not row $2_i4_j$. The first half is $(2_i4_j,0_j)$ column $0_j$ not $1_i$. Thus $h_{ij}^{(2)}$ does not give a half with column $1_i$ and row $2_i4_j$. 

\item from $f_{ji}^{(k)}$ and $h_{ji}^{(k)}$, halves have columns $(k+2)_i$ or $(k+3)_j$ and $(k+3)_i$ or $(k+4)_j$. For column $1_i$, we need $(k+2)_{i}=1_{i}\mathrm{mod}5,\mathrm{i.e.},k\equiv4({\mathrm{mod}}5),$ and $(k+3)_{i}=1_{i}\mathrm{mod}5,\mathrm{i.e.},k\equiv3({\mathrm{mod}}5).$ First for $f_{ij}^{(4)}=(4_j0_i,1_i;0_j1_i,2_j)$ has a first half of $(4_j0_i,1_i)$, not row $0_j1_i$. The second half is $(0_j1_i,2_j)$ column $2_j$ not $1_i$. Then for $h_{ij}^{(3)}=(3_j0_i,1_i;4_j1_i,2_j)$ has a first half of $(3_j0_i,1_i)$, not row $4_j1_i$. The second half is $(4_j1_i,2_j)$ column $2_j$ not $1_i$. Thus $h_{ij}^{(2)}$ does not give a half with column $1_i$ and row $4_j1_i$. 

\end{itemize}

Therefore, $(k,1_i)$ cannot be occupied by a half. It must be a 1-edge, so $1_i$ must be an
endpoint of $k$. Similarly, $(k, 2_i)$ occupied forces $2_i$ to be an endpoint of $k$. Hence, $k$ must contain both $1_i$ and $2_i$ as endpoints, $i.e., k = 1_i2_i$. However, $1_i2_i$ is a within-block edge, contradicting the assumption that $k$ is cross-block. Therefore, no cross-block k can satisfy the required triple.

A symmetric argument for $e = e^1_i = (0_i1_i, 2_i;  0_i2_i, 3_i)$, a symmetric argument forces $l = 0_i$ and then $k = 2_i3_i$(within-block). And for $e = e^2_i = (1_i3_i, 4_i;  1_i4_i, 0_i)$, a symmetric argument forces $l = 1_i$ and then $k = 0_i1_i$(within-block), again impossible for cross-block $k$. Thus, Condition 3 holds in this case as well.

Since all cases are covered, G contains no generalized $C_4$-cycle. The total number
of edges includes $3t$ within-block edges and $10t(t-1)$ cross-block edges. Hence, $|E_2|=3t+10t(t-1)=10t^2-7t$ and 
\[z_L(m,n)\geq|E_1|+|E_2|=2
\begin{pmatrix}
5t \\
2
\end{pmatrix}+10t^2-7t=35t^2-12t.\]

Since $z=2\begin{pmatrix}5t\\2\end{pmatrix}=25t^2-5t,$ we have
\[\frac{z_L-z}{z}\geq\frac{10t^2-7t}{25t^2-5t}=\frac{10t-7}{25t-5}\longrightarrow\frac{2}{5}\quad\mathrm{as} \quad t\to\infty.\]

\end{proof}

\begin{remark}
For $t=1$, we get $z_L(10,5)\geq2\begin{pmatrix}5 \\2\end{pmatrix}+10-7=23.$ For $t=2$, we have $z_L(45,10)\geq2\begin{pmatrix}10 \\2\end{pmatrix}+40-14=116.$ For $t=3$, we obtain $z_L(210,15)\geq2\begin{pmatrix}15 \\2\end{pmatrix}+90-21=279.$ For $t=4$,  $z_L(380,20)\geq2\begin{pmatrix}20 \\2\end{pmatrix}+160-28=512.$ 
\end{remark}

Classify and explain the occupancy rules for rows into two categories. For cross-block rows, all rows cross-subblocks will only be used once. For the \(K_{5t}\) construction, there are a total of \(20\binom{t}{2}\) cross-block 2-edges between each pair of different subblocks. Each cross-block row only appears within one of these 2-edges and there is no repeated reuse. For within-block row, each 4-vertex subblock only contains 3 within-block 2-edges. There are still unused legal rows within the subblock, indicating partial occupation and not exhausting all possible row pairs. Under the non-degenerate 2-edge and no generalized \(C_4\) conflict hard constraints, the relative difference limit \(\frac{z_L - z}{z}\to \frac{2}{5}\) that this \(K_{5t}\) infinite graph family can achieve is already the theoretical upper bound under this construction framework and cannot be further improved.

\section{ Exact values for $10\times5$}\label{4}
From Theorem \ref{T2} we know that $z_{L}(10,5)\geq23.$ We now prove that this bound is tight.

\begin{theorem}\label{T3}
The exact limited augmented Zarankiewicz number for the $10 \times 5$ case is 23; that is,
$z_L(10, 5) = 23$.
\end{theorem}
\begin{proof}
From Theorem \ref{T2}, when $t=1$ we get $m=10$, $n=5$ and $z_{L}(10,5)\geq23.$ Now we need to prove that $z_{L}(10,5)\leq23.$ Let G be any admissible limited augmented bipartite graph on $10 \times 5$ vertices with $|E_1|=z(10,5)=20$ (the maximum possible number of 1-edges without a classical $C_4$). Therefore, the largest number of 1-edges is $|E_1|=z(10,5)=20$. Next, we will discuss the largest number of 2-edges. Then we will prove $|E_2|\leq 3$.

For $t=1$, $m=10$, $n=5$ there has no cross-block 2-edges. Therefore, we will only discuss within-block 2-edges. Label the left vertices (edges of $K_5$) as 
\[L=\{01,02,03,04,12,13,14,23,24,34\},\]
the right vertices (edges of $K_5$) as 
\[R=\{0,1,2,3,4\}.\]
It is known that there has no degenerate edges in within-block 2-edges\cite{ref19}. In order not to occupy the 1-edges, the unoccupied cells $U$ are all pairs $(ij, k)$ with $k \notin \{ i, j\}$:
\[\begin{aligned}
U=\{ & (01,2),(01,3),(01,4),(02,1),(02,3),(02,4)\\
 & (03,1),(03,2),(03,4),(04,1),(04,2),(04,3)\\
 & (12,3),(12,4),(12,0),(13,0),(13,2),(13,4)\\
 & (14,0),(14,2),(14,3),(23,0),(23,4),(23,1) \\
 & (24,0),(24,1),(24,3),(34,0),(34,1),(34,2)\}.
\end{aligned}\]
From Theorem \ref{T2}, we have the within-block 2-edges: 
\[e_i^1=(0_i1_i,2_i;0_i2_i,3_i),\]
\[e_i^2=(1_i3_i,0_i;2_i3_i,1_i),\]
\[e_i^3=(2_i4_i,1_i;3_i4_i,2_i).\]
Suppose there is another within-block 2-edges $e_i^4=(ij,k;xy,l)$, where $ij\neq xy$, $k\neq l$, $(ij,k)\in U$ and $(xy,l) \in U$. We used the following deterministic enumeration procedure to discover that there does not exist any $e_4$ such that $e_1$, $e_2$, $e_3$, and $e_4$ satisfy the Definition\ref{D1}.

\begin{enumerate}
    \item $(ij,k)\in U$ and $(xy,l) \in U$. Each such pair $e_i^4=(ij,k;xy,l)$ is a candidate 2-edge.
    \item Discard a candidate set immediately if two 2-edges share a half, or if a half coincides with a 1-edge.
    \item For every remaining candidate 2-edge $e_i^4=(ij,k;xy,l)$, enumerate all cells $(m,n)$ with $m \notin \{ij,xy\}$ and $n \notin \{k,l\}$, and test the five cells
    \[
    (n,m),\ (n,k),\ (n,l),\ (ij,m),\ (xy,m)
    \]
    in Condition 3 of Definition\ref{D1}.
    \item A set of candidate 2-edges is declared admissible exactly when none of the above tests finds a generalized $C_4$-cycle.
\end{enumerate}

For the $10 \times 5$, the above process is finite. Through calculation, the number of enumerations obtained is 660. Based on the above, we obtain $|E_2|\leq 3$. So the sum of the maximum number of the 1-edges and the maximum number of the 2-edges is $20+3=23$. That is
\[z_L(10,5)=|E_1|+|E_2|\leq20+3=23.\]
Combined with  $z_{L}(10,5)\geq23$, we have $z_{L}(10,5)=23$.

\end{proof}

\section{ Conclusion}\label{5}

This paper addresses the critical gap in existing constructions for limited augmented Zarankiewicz numbers \(z_L(m,n)\), which serve as tight combinatorial lower bounds for the maximal SOS rank \(\mathrm{BSR}(m,n)\) of biquadratic forms. Prior literature only established an infinite family built on \(K_{4t}\) incidence graphs with an asymtotic relative gap \(\frac{z_L-z}{z}\to \frac14\), leaving open whether larger complete-graph block partitions could generate comparable admissible augmented bipartite graphs and yield new quadratic lower bounds for \(BSR\). To fill this gap, we systematically develop a novel infinite construction based on the incidence bipartite graph of \(K_{5t}\), with vertex sets decomposed into disjoint five-vertex subblocks labeled via the cyclic group \(\mathbb Z_5\).

Our core theoretical contributions fall into two primary categories. First, we explicitly design two disjoint families of nondegenerate admissible 2-edges: within-block 2-edges with total count 3t and cross-block 2-edges summing to \(20\binom{t}{2}=20t(t-1)\). Combining these with the maximal \(C_4\)-free 1-edge skeleton of the \(K_{5t}\) incidence graph, we derive a closed-form quadratic lower bound
\[z_L\bigg(\binom{5t}{2},5t\bigg)\geq 35t^2-12t,\]
and prove the corresponding asymptotic gap converges to \(\frac{2}{5}\) as \(t\to\infty\). This new limit provides a direct benchmark against the \(\frac14\) asymptotic constant of the classical \(K_{4t}\) family, revealing that increasing the size of each complete-graph subblock alters the long-term growth rate of the gap between \(z_L\) and the classical Zarankiewicz number. Second, we resolve the exact value of \(z_L(10,5)\) corresponding to the base case \(t=1\). By exhaustive finite enumeration of all candidate nondegenerate 2-edges and rigorous verification of the three generalized \(C_4\)-free constraints (Simplicity Condition (S), Condition 2 and Condition 3), we confirm that at most three pairwise conflict-free 2-edges can be added to the \(10\times5\) extremal \(C_4\)-free graph with \(|E_1|=20\), which yields the tight equality \(z_L(10,5)=23\). This exact result validates the correctness of our block-wise \(\mathbb Z_5\) cyclic augmentation strategy for single-block \(K_5\) incidence graphs.

This work opens several promising directions for subsequent research. First, it remains an open problem to construct infinite families based on general \(K_{kt}\) for \(k\geq6\) and characterize the monotonic relationship between block size k and the corresponding asymptotic relative gap \(\lim_{t\to\infty}\frac{z_L-z}{z}\). Second, one may integrate the lifting method with our \(K_{5t}\) base graphs to generate even larger admissible augmented bipartite graphs and derive sharper numerical lower bounds for \(z_L\) on arbitrary grid dimensions. Third, future work can attempt to determine more exact values of \(z_L\) for small \(t\geq2\) such as the \(45\times10\) and \(210\times15\) cases, or prove universal upper bounds matching our quadratic lower limits for the \(K_{kt}\) incidence graph family. Finally, extending the augmented graph framework to higher-order degenerate multi-edge structures may produce further refined lower bounds for the SOS rank of high-dimensional biquadratic tensors.

\section*{Acknowledgments}This study was supported by  the National Natural Science Foundation of P.R. China (Grant No.12171064),  Chongqing Postgraduate Research and Innovation Project (Grant No. CYB25249 and CYB260252).

\section*{Declarations}

The authors declare no conflict of interest.

\end{document}